\documentclass[12pt]{article}
\usepackage{amsmath,amsthm,amsfonts,amssymb,mathrsfs,bm}
\usepackage[usenames]{color}
\usepackage{hyperref}
\usepackage{amssymb}
\parskip 	\bigskipamount

\newtheorem{theorem}{Theorem}[section]

\newtheorem{corollary}[theorem]{Corollary}
\newtheorem{proposition}[theorem]{Proposition}

\newtheorem{definition}{Definition}

\def\Z{\mathbb{Z}}
\def\R{\mathbb{R}}
\def\C{\mathbb{C}}

\def\E{\mathbb{E}}
\def\P{\mathbb{P}}

\def\td{\mathbb{T}^d}
\def\rd{{\R^d}}
\def\t{\mathbb{T}}
\def\wt{\widetilde}
\def\th{\theta}
\def\S{\mathcal{S}}
\def\b{\beta}

\def\O{\Omega}

\def\ph{\varphi}
\def\var{\mathrm{Var}}
\def\cov{\mathrm{Cov}}
\def\sin{\mathrm{sin}}

\def\eps{\epsilon}
\def\del{\delta}
\def\h{\hat}

\def\G{\Gamma}
\def\u{\Upsilon}
\def\c{\complement}
\def\d{\mathrm{d}}

\def\inn{\mathrm{in}}
\def\out{\mathrm{out}}

\renewcommand{\l}[0]{\left }
\renewcommand{\r}[0]{\right}

\makeatletter
\renewcommand*{\@cite@ofmt}{\hbox}
\makeatother

\begin{document}
\title{Generalized stealthy hyperuniform processes : maximal rigidity and the bounded holes conjecture }
\author{
	\begin{tabular}{c}
		{Subhro Ghosh}\\  
		 Princeton University\\ subhrowork@gmail.com
	\end{tabular}
	\and
	\begin{tabular}{c}
		{Joel L. Lebowitz}
		 \\ Rutgers University\\ lebowitz@math.rutgers.edu
	\end{tabular}
}
%\author{\href{httpRevision://www.princeton.edu/~sg18/}{Subhroshekhar Ghosh}\\ Princeton University \\ sg18@princeton.edu}
\date{\today}
\maketitle

\begin{abstract}

We study translation invariant stochastic processes on $\R^d$ or $\Z^d$ whose diffraction spectrum or structure function $S(k)$, i.e. the Fourier transform of the \textsl{truncated} total pair correlation function, vanishes on an open set $U$ in the wave space. A key family of such processes are ``stealthy'' hyperuniform point processes, for which the origin $k=0$ is in $U$; these are of much current physical interest. We show that all such processes exhibit the following remarkable “maximal rigidity” : namely, the configuration outside a bounded region determines, with probability 1, the exact value (or the exact locations of the points) of the process inside the region. In particular, such processes are completely determined by their tail. In the 1D discrete setting (i.e. $\Z$-valued processes on $\Z$), this can also be seen as a consequence of a recent theorem of Borichev, Sodin and Weiss (\cite{BSW}); in higher dimensions or in the continuum, such a phenomenon seems novel. For stealthy hyperuniform point processes, we prove the Zhang-Stillinger-Torquato conjecture (\cite{Bounded}) that such processes have bounded holes (empty regions), with a universal bound that depends inversely on the size of $U$. 

\end{abstract}

\section{Introduction}

In recent years, a special class of hyperuniform particle systems, known as stealthy hyperuniform (henceforth abbreviated as SH) systems, have attracted considerable attention (\cite{Stealthy-1}, \cite{Stealthy-2}, \cite{Stealthy-3}, \cite{Stealthy-4}, \cite{Stealthy-5}). These systems are characterized by the structure function $S(k)$ vanishing in a neighbourhood of $k=0$. A natural generalization of SH point processes is to consider point processes, or random fields, with a gap in the spectrum on an open set $U$ which may not include the origin. We shall denote these processes as Generalized Stealthy (henceforth abbreviated as GS) processes.

The nomenclature ``stealthy'', as well as the physical interest in SH particle systems, stems from the fact that such systems are optically transparent (invisible) for wave vectors $k$ in the gap $U$. Numerical and experimental investigations have been carried out regarding how to construct  SH particle systems. These systems cannot be equilibrium systems, with tempered potentials, at finite temperatures. They may, however, be ground states of such systems, e.g. the periodic (disordered ?) zero temperature states of classical systems, or they can be generated as non-equilibrium states. 

SH systems are an extension of hyperuniform (superhomogeneous) particle systems. Hyperuniform systems, which have been studied extensively both in the physics and the mathematics literature (\cite{AM}, \cite{MY}, \cite{GLS}, \cite{GL}, \cite{ToSt}, \cite{Glass}, \cite{Avian}, \cite{Jammed},  \cite{Photonic}, \cite{Neuronal}, \cite{Expt-2}, \cite{Expt-3}, \cite{HL}, \cite{HCL}), have reduced fluctuations: the variance of the particle number in a domain $D$ in $\R^d$ or $\Z^d$ grows slower than the volume of $D$. The significance of hyperuniform materials, and in particular SH systems, lies in the fact that they embody properties of both crystalline and disordered or random systems.  (see \cite{Torq}, \cite{GL} and the references therein). 

For translation invariant systems, an equivalent characterization of hyperuniformity can be obtained by looking at its structure function. Hyperuniformity then boils down to the vanishing of the structure function $S(k)$ at $k=0$. SH systems, therefore, involve a specific manner in which this vanishing of the structure function takes place.

In particle systems theory, an important object of investigation is the event that there is a large 	``hole'' in the particle configuration, where, by a ``hole'' we generally mean a ball, with centre at the origin, that is devoid of any particle. 
The rate of decay of the hole probability encodes important structural information about the particle system at hand. Heuristically speaking, the faster this rate of decay the stronger is the crystalline (or lattice-like) behaviour of the particle system. 

The examination of hole probabilities for stealthy hyperuniform processes, at the numerical level, brings up intriguing challenges. In particular, numerical studies led to the conjecture that the hole sizes for stealthy hyperuniform processes are, in fact, uniformly bounded. Furthermore, it was conjectured that the upper bound on the hole size depends only on the size of the gap for the structure function of the process. 
In the article \cite{Bounded}, Zhang, Stillinger and Torquato provide numerical evidence in support of these remarkable conjectural properties of stealthy hyperuniform processes. Note that this is definitely not true for all hyperuniform systems, as exemplified by the one component Coulomb systems, in particular by the Ginibre ensemble (\cite{JLM}).

In this paper, we carry out a rigorous mathematical analysis of stealthy hyperuniform processes, and establish the veracity of this conjecture. In particular, we prove that 
\begin{theorem}
\label{bounded-holes}
	Let $\Xi$ be a stealthy hyperuniform point process. Let $B(x;r)$ be the ball with centre $x$ and radius $r$.  Then there exists a positive number $r_0$ such that $\P[|\Xi \cap B(x;r_0)|=0]=0$. 
\end{theorem}

Furthermore, we establish the explicit dependence of $r_0$ on the size of the gap $b$ in the spectrum, and show that it satisfies an inverse power law.
\begin{theorem}
\label{inverse}
The quantity $r_0$, as in the statement of Theorem \ref{bounded-holes}, can be chosen to be $Cb^{-1}$, where $b$ is the radius of the maximal ball (centered at the origin) that is contained in the gap of the structure function $S$, and $C$ is a universal constant. 
\end{theorem}

The proofs of Theorems \ref{bounded-holes} and \ref{inverse} exploit an anti-concentration property for particle numbers of stealthy hyperuniform processes:
\begin{theorem}
\label{anti-conc}
Let $\Xi$ be a stealthy hyperuniform point process on $\R^d$ or $\Z^d$ with one point intensity $\rho$ and $b$ the radius of the  largest ball around the origin (in the wave space) on which the structure function of $\Xi$ vanishes. There exists numbers $C,c>0$ (independent of all  parameters of $\Xi$) such that,  the number of points of $\Xi$ in any given $d$-dimensional cube of side-length $Cb^{-1}$ is a.s. bounded above by $c\rho b^{-d}$.
\end{theorem}

The fact that holes in SH processes cannot be bigger than a deterministic size is suggestive of a high degree of crystalline behaviour in these processes. In this paper, we go further, and establish a remarkable maximal rigidity property of these ensembles. We can, in fact, do this in the setting of GS processes. Before stating our main result, we need to discuss the concepts of rigidity and maximal rigidity. This we will do in the context of $\rd$, noting that these concepts in the setting of $\Z^d$ are completely analogous.

For a point process (more generally, a random field or a random measure) $\Xi$ on $\rd$ and a bounded domain $D \subset \rd$,  statistic $\Psi$ defined on $\Xi$ restricted to $D$ is said to be rigid if $\Psi$ is completely determined by (that is, a deterministic function of) the process $\Xi$ restricted to $D^\c$. To put things in perspective, a point process having rigidity is in notable contrast to the Poisson process, where the process inside and outside of $D$ are statistically independent.  Rigidity phenomena for particle systems have been investigated quite intensively in the last few years, and it has been shown to appear in many natural models which are, nonetheless, far removed from being crystalline. Key examples include  the Ginibre ensemble, Gaussian zeros, the Dyson log gas, Coulomb systems and various determinantal  processes related to random matrix theory (see, e.g., \cite{GP}, \cite{G1}, \cite{Bu}, \cite{OsadaSh}, \cite{GL}).  In \cite{G2}, rigidity phenomena were shown to be intimately connected to mutual singularity properties of various Palm measures of a point process, which is a topic of independent interest in particle systems theory. In \cite{GK}, a one parameter family of Gaussian random series were introduced, whose zeros exhibit rigidity of increasing number of  moments, as the parameter is allowed to vary.

In this paper, we show that  GS random measures on $\rd$ or $\Z^d$ exhibit maximal rigidity : namely, for any domain $D \subset \rd$,  the random measure $[\Xi]_{| D^\c}$ determines completely the measure $[\Xi]_{| D}$ (that is, the latter is a deterministic measurable function of the former). Stated in formal terms, we prove:

\begin{theorem}
\label{rigidity}
Let $\Xi$ be a generalized stealthy random measure on $\rd$ or $\Z^d$. Then for any bounded domain $D$, the random measure $[\Xi]_{| D}$ is almost surely determined by (i.e., is a measurable function of) the random measure $[\Xi]_{| D^\c}$.
\end{theorem}

This leads to the following conclusion:

\begin{corollary}
\label{tail}
Let $\Xi$ be a generalized stealthy random measure on $\rd$ or $\Z^d$. Then $\Xi$ is completely determined by (that is, measurable with respect to) its tail sigma field. 
\end{corollary}

We further show that, to have maximal rigidity in the sense discussed above, it suffices that the structure function vanishes faster than any polynomial at some point in the wave space:

\begin{theorem}
\label{fast-decay}
Let $\Xi$ be a  random measure on $\rd$ or $\Z^d$ such that the structure function of $\Xi$ vanishes faster than any polynomial at the origin in the wave space. Then  for any bounded domain $D$, almost surely the random measure $[\Xi]_{| D}$ is determined by (i.e., is a measurable function of) the random measure $[\Xi]_{| D^\c}$. Consequently,  $\Xi$ is completely determined by (that is, measurable with respect to) its tail sigma field. 
\end{theorem}

In the special case of the 1D discrete setting, that is $\Z$-valued processes on $\Z$, our results can also be seen as a consequence of a remarkable recent theorem of Borichev, Sodin and Weiss \cite{BSW}. They show that, in the 1D discrete setting, the Fourier transform of the covariance function vanishing on a sub-interval of $\t$ (which is the wave space in this case) implies that the process is essentially periodic. However, the methods of \cite{BSW}, which rely strongly on complex analysis of one variable, seem to be rather specially suited for the 1D discrete case, and in \cite{BSW} the authors raise the question as to whether one can expect similar results in the discrete setting in higher dimensions (that is, $\Z$-valued processes on $\Z^d$ for $d > 1$). This leaves out of its purview point processes in the continuum in any dimension (that is, on $\rd$ for $d \ge 1$). For continuum random fields or random measures on $\Z^d$ or $\rd$, periodicity is not true for stealthy processes as evidenced by the key example of Gaussian processes, already noted in \cite{BSW} and discussed in Section \ref{sec:Gaussian process}. This is an open question for particle systems; see \cite{Stealthy-1} for physical arguments and numerical evidence regarding the existence of SH disordered particle systems.

The question of inference about a stochastic process from its diffraction spectrum has a long history in  diffraction theory, and we believe the results in the present article would be of interest to that body of literature. For the reader interested in further exploration of this direction, we refer to the survey \cite{KinView} and the more recent article \cite{ContDiff}.

\section{Setup and basic notions}
\label{setmod}
In this section, we formally define the setup in which we are working.

At the very outset, let us introduce the notation that $\Omega$ denotes the wave space. This would imply that $\Omega=\td$ when the physical space is $\Z^d$, and $\Omega=\rd$ when the physical space is $\rd$. The value of $\Omega$ will usually be clear from the context. On a related note, we will denote functions that live on the wave space with a cap, and we will denote functions that live on the physical space without a cap.

Let $\S$ denote the space of complex Borel measures on $\rd$ (resp., $\Z^d$). Then a random measure on $\rd$ is a random variable that takes values in the space $\S$. Since a locally finite point configuration can be identified with its counting measure, a point process is a special case of a random measure. It also covers locally $L_1$ random fields on $\rd$, by identifying a random field $\{X_u\}_{u \in \rd}$ with the random measure $X_u \d u$ (where $\d u$ is the Lebesgue measure on $\rd$). If $\Xi$ is a random measure, then by $\d [\Xi]$ we will denote a specific realization of the random variable $\Xi$.

A random measure is characterized by the distribution of its linear statistics, namely the random variables $\int \ph(x) \d[\Xi](x)$ for $C_c^\infty$ test functions $\ph$ on $\rd$. A random measure $\Xi$ is called translation invariant if, for any $u \in \rd$, the distribution of the measure $\d [\Xi] (\cdot-u)$ is the same as that of $\d [\Xi](\cdot)$.

Translation invariance of $\Xi$ implies that $\E\l[\d [\Xi]\r]=\rho \d u$, where $\rho$ is the intensity of the random measure and $\d u$ denotes the Lebesgue measure.

The covariance functional $\beta$ of a random field $\Xi$ is defined by the equation, for any two test functions $\ph$ and $\psi$:
\[  \cov \l[ \int \ph \d [\Xi], \int \psi \d [\Xi] \r]  = \int \int \ph(s) \psi(t) \beta(s-t) \d s \d t. \] In general,  the covariance functional $\d \beta(s,t)$ will be a measure (that is a constant times the Lebesgue measure on the sets $s-t=\mathrm{const.}$ because of translation invariance).  We assume here for simplicity that this measure, in fact, has a density with respect to the Lebesgue measure $\d s \d t$ on $\rd \times \rd$.

As a special case of this, for a translation invariant point process $\Xi$, we denote by $G$ the fully truncated pair correlation  function, given by 
\[G(x-y)=\rho_2(x-y) - \rho^2 + \rho \del(x-y),\] 
where $\rho_2$ is the two point correlation function of $\Xi$ and $\rho$ is its (one point) intensity. It can be verified via a simple computation that $G$ is the same as $\beta$ as discussed above.

The Fourier transform $\h{\beta}$ of the covariance functional $\beta$ is the structure function $S(k)$ (or, more generally, the structure measure $\d S(k)$). 

%A random measure is called stealthy if its structure function  (or, more generally, its structure measure) vanishes in a neighbourhood of the origin.

We now formally define hyperuniform and stealthy hyperuniform random measures.

\begin{definition}
	\label{hyperunif}
A point process is called \textsl{hyperuniform} if the structure function $S$ vanishes at the origin.
\end{definition}
Equivalently, it can be shown that a point process is hyperuniform if and only if the variance of the particle number in a domain $D$ grows slower than the volume $|D|$ (and not like the volume, as would be the case for Poisson and other extensive point processes).  

We will focus on \textsl{Stealthy} hyperuniform processes, which are defined as follows:
\begin{definition}
	\label{stealthy}
	A point process is called \textsl{stealthy hyperuniform} (SH)  if the structure function $S$ vanishes in a neighbourhood of the origin.
\end{definition} 

Finally, we have \textsl{Generalized Stealthy} processes:
\begin{definition}
\label{gensth}
A point process (or random measure) on $\rd$ (resp., $\Z^d$) is called a generalized stealthy (GS) process if its structure function $S$ vanishes in an open set $U \subset \rd$ (resp. $\td$).
\end{definition}
The crucial difference between stealthy and generalized stealthy processes is that, for a generalized stealthy process, the gap $U$ in the diffraction spectrum need not contain the origin.

\section{Universal bounds on hole sizes}
\label{hole-size}
Our goal in this section is to prove Theorem \ref{bounded-holes}. Let $\Xi$ be a stealthy process on $\R^d$ or $\Z^d$, whose structure function vanishes in a neighbourhood $U$ of the origin in the wave space. Let $b$ denote the radius of the largest ball, with centre at the origin, that is contained in $U$.
% Let $U$ be the maximal ball in the frequency space with centre 0 such that the structure function vanishes on $U$. Let $b$ be the radius of $U$.
% The arguments for the discrete ($\Z^d$) and the continuum ($\R^d$) cases are somewhat different, so we discuss  them in two consecutive sub-sections.

\subsection{Linear statistics}
\label{linstat}
We consider an $L_2$ function  $\ph$, and the linear statistic $I(\ph)=\int \ph \d [\Xi]$ corresponding to $\ph$. We consider the variance of the linear statistic which, for translation invariant point processes, can be written as
\begin{equation}
\var[I(\ph)]
%= & \rho \sum_{x \in \Z^d} |\ph(x)|^2  + \sum_{x, y \in \Z^d} \ph(x) \ol{\ph(y)} \rt(x-y)  \\
=  \int_\O |\h{\ph}(\xi)|^2 S(\xi) \d \xi.
\end{equation}
Now, if $S(\xi)$ vanishes in a neighbourhood $U \subset \Omega$, and if $\ph$ is such that  $\h{\ph}$ is supported on $U$, we have $\var[I(\ph)]=0$. This implies that a.s. we have
\begin{equation}
\label{fund}
\int \ph \d [\Xi]  = \E[I(\ph)] = \rho \h{\ph}(0).
\end{equation}
This is a crucial observation in the context of stealthy processes with important implications, as we shall see in the ensuing sections.

\subsection{Anti-concentration of particles in stealthy processes}
\label{antic}
In this section, we show a rather surprising ``anti-concentration'' property of the particles in a stealthy hyperuniform point process. Though it arises contextually in the proof of boundedness of holes in such processes, we believe that this property is interesting in its own right.

\begin{proof}[Proof of Theorem \ref{anti-conc}]
We observe that, since the process $\Xi$ is translation invariant, it suffices to consider the cube (in which we want to prove the anti-concentration) to be centered at the origin.

We first choose an auxiliary test function $\h{\psi}$ on $\O$ such that
 $\h{\psi}$ is a $C_c^\infty$ non-negative radial function supported on the ball of radius $1/3$ centered at the origin. This has several consequences, which we enumerate below: 
% such that $2\cdot B \subset U$, with $\h{\psi} \ne 0$. 
\begin{itemize}
\item Since $\h{\psi}(\xi)$ is invariant under the map $\xi \mapsto -\xi$, we deduce that $\psi$ is real valued. \item Since $\h{\psi}$ is a radial function, so is $\psi$.  
\item Since $\h{\psi}$ is non-negative, we have $\psi(0)=\int \h{\psi}(\xi) \d \xi >0$. 
\item Since $\h{\psi}$ is $C_c^\infty$, we deduce that $\psi$ is a Schwarz function, and hence in fact $\psi>0$ on a neighbourhood of the origin.  
\item Dividing $\h{\psi}$ by a positive constant if necessary, we may assume that $\psi \ge 1$ on a cube of side $a>0$  centered at the origin.
\end{itemize} 
 
 We now select $\ph$ to be the function such that $\h{\ph}(\xi)=(\h{\psi} * \h{\psi})(\xi/b)$. Clearly, $\ph(x)=b^d[{\psi}(bx)]^2$.  

%Then we choose $\ph$ by setting $\h{\ph}=\h{\psi}*\h{\psi}$. 
We observe that :
\begin{itemize}
\item $\h{\psi} * \h{\psi}$ is supported inside the ball of radius $2/3$ (centered at the origin) in the wave space. Consequently, $\h{\ph}$ is supported in the ball of radius $\frac{2}{3}b$ centered at the origin, and hence its support $\subset U$ (which contains a ball having center at the origin and radius $b$).
\item $\ph$ is a non-negative radial Schwarz function.
\item $\h{\ph}(0)=\int \ph(x) \d x )>0$.
\item $\ph \ge b^d$ on a cube $B$ centred at the origin and having side length $ab^{-1}$, where $a$ is as defined while enumerating the properties of $\psi$.
\end{itemize}
Due to the fact that the support of $\h{\ph}$ is contained inside $U$, the function $\ph$ satisfies \eqref{fund} a.s. This, together with the fact that $\ph \ge 0$, implies that 
\[ \sum_{x \in \Xi \cap B} \ph(x) \le \rho \h{\pi}(0)=\h{\psi} * \h{\psi}(0). \]
This, in turn, implies that 
\[ |\Xi \cap B| \cdot b^d \le \rho \h{\psi} * \h{\psi}(0).  \]
Consequently, we have
\[  |\Xi \cap B|  \le \h{\psi} * \h{\psi}(0) \cdot  \rho b^{-d},\] as desired.

\end{proof}

\subsection{Proof of Theorem \ref{bounded-holes}}
\label{bdd-sec}
In this section, we use the anti-concentration results of the previous section in order to establish boundedness of holes for stealthy hyperuniform point processes.

For the rest of this section, we will work on the event (having probability 1) where anti-concentration of particle number holds simultaneously for all cubes of side length $ab^{-1}$ having  centers with co-ordinates that are rational multiples of $\sqrt{d}$. 

In what follows, we will denote by $A(r,s)$ the $L_\infty$ annulus $\{ x \in \rd : r \le \|x\|_\infty <s \}$.
By $\theta$ we will denote $Cb^{-1}$, where $C$ and $b$ are as in the statement of Theorem \ref{anti-conc}.
% For $x \in \rd$ and $r>0$, we will also denote by $\G(x,r)$ the cube with centre at $x$ and side length $r$. 
%Let $b$ be the side length of the cube $B$.  
Then for parameter$R>0$, to be thought of as a large parameter, we will consider the decomposition \[ \rd = A(0,R \theta) \bigcup \cup_{j=1}^\infty A(R\theta+j\theta,R\theta+(j+1)\theta). \] 
In the remainder of this section, we will denote $ A(0,R \theta) $ by $A_0$ and $A(R\theta+j\theta,R\theta+(j+1)\theta)$ by $A_j$ for $j \ge 1$.

We consider the equation \eqref{fund} for $\ph$ (chosen as in Section \ref{antic}), and observe that we can rewrite the same as 
\begin{equation}
\label{fund-cont-re}
\l( \sum_{x \in \Xi \cap A_0} \ph(x) \r) + \sum_{j=1}^{\infty} \l( \sum_{x \in \Xi \cap A_j} \ph(x) \r) = \rho \h{\ph}(0). 
 \end{equation}
Let $R$ be such that the event $F_R:=\{ |\Xi \cap A_0|=0 \}$ has  positive probability. Then, on the event $F_R$, \eqref{fund-cont-re} reduces to 
\begin{equation}
\label{fund-cont-re1}
\sum_{j=1}^{\infty} \l( \sum_{x \in \Xi \cap A_j} \ph(x) \r) = \rho \h{\ph}(0). 
\end{equation}
Since $\ph$ is a Schwarz function, we have $|\ph(x)| \le C_\ph |x|^{-d-1}$. 
We now proceed to understand the quantity $C_\ph$ for our choice of $\ph$. We begin with the observation that $C_\ph= \| |x|^{d+1} \ph \|_\infty$. Going to the Fourier space, this can be bounded above as
\[ C_\ph  = \||x|^{d+1} \ph \|_\infty \le  \|D^{d+1} (\h{\ph})\|_1, \] where $D$ is the radial derivative. But recall that  \[\h{\ph}(\cdot)=(\h{\psi} * \h{\psi})(\cdot/b),\] so $\|\h{\ph}\|_1=b^d\|\h{\psi} * \h{\psi} \|_1$, and $\|D^{d+1}\h{\ph}\|_1=b^{-1}\|D^{d+1}(\h{\psi} * \h{\psi})\|_1$. Here we recall the fact that $\psi$ is a universal choice (that is, not dependent on any parameters of the process $\Xi$), and hence all norms of $\h{\psi} * \h{\psi}$ are universal constants in the context of our  problem.
%Arguing as in Section \ref{scale-disc}, since $\ph$ is a Schwarz function,
Putting together the above observations, we have $|\ph(x)| \le C b^{-1} |x|^{-d-1}$, where $C$ is a universal constant. Consequently, 
\[ \l| \sum_{x \in \Xi \cap A_j} \ph(x) \r| \le C b^{-1} |\Xi \cap  A_j|[(R+j)\theta]^{-d-1} , \]
which implies that on the event $F_R$ we have
\begin{equation}
\label{cont-fund-re2}
0<\rho \h{\ph}(0) \le \sum_{j=1}^\infty |\Xi \cap  A_j|[(R+j)\theta]^{-d-1}
\end{equation}

%Since $\sum_{j=1}^{\infty}(Rb+jb)^{-d-5} \le C (Rb)^{-3}$, we deduce from \eqref{fund-cont-re2} that, for $R$ large enough (depending on $\ph$), we must have \[|\Xi \cap A(Rb+jb,Rb+(j+1)b)| \ge C (Rb+jb)^{d+5}  \]
%for some $j \ge 1$. Fix such a $j=j_0$. But
Let  $\G(x,\theta)$ denote the cube with centre $x \in \R^d$ and side-length $\theta$.
We can then write  $A_j= A((R+j)\theta,(R+(j+1))\theta)$ as a disjoint union
\[ A_j = \cup_{i=1}^N \G(x_i,\theta),\]  such that the centres $x_i$ have co-ordinates that are rational multiples of $\sqrt{d}$ and $N=C[(R+j)\theta]^{d-1}$ for some positive number $C$. 
%This implies that, for some $i$, we must have  \[ |\Xi \cap \G(x_i,b)| \ge C b^{d+5} (R+j)^6.  \] Fix such an $i=i_0$.

%We now consider \eqref{fund-cont} for the function $\ph_{x_{i_0}}$. By the properties of $\ph_{x_{i_0}}$, we have $\ph_{x_{i_0}} \ge 1$ on the set $\G(x_{i_0},b)$. Since $\ph_{x_{i_0}}$ is non-negative, we have
%\[\sum_{x \in \Xi} \ph_{\xi_{i_0}}(x) \ge \sum_{x \in \Xi \cap \G(x_{i_0},b)} \ph_{x_{i_0}}(x) \ge  |\Xi \cap \G(x_{i_0},b)| \ge C b^{d+5} (R+j)^6.  \]
%By \eqref{fund-cont} applied to  $\ph_{x_{i_0}}$, we deduce that 
%\begin{equation}
%\label{contra}
%\rho \h{\ph}_{x_{i_0}}(0) = \rho \h{\ph}(0) \ge  C b^{d+5} (R+j)^6. 
%\end{equation}
%But by choosing $R$ large enough (depending on $\ph$), we can ensure that \eqref{contra} is false for all $j \ge 1$. 
But the anti-concentration property of particle numbers (Theorem \ref{anti-conc}) implies that $|\Xi \cap \G(x_i, \theta)| \le c \rho b^{-d}$. This implies that 
\begin{equation}
\label{cont-fund-re3}
0<\rho \h{\ph}(0) \le c b^{-1} \sum_{j=1}^\infty  \rho b^{-d}  [(R+j)\theta]^{-d-1} \le c_1 b^{-d-1} \rho R^{-1} \theta^{-d-1} = c_2 \rho R^{-1}.  
\end{equation}
Consequently, $R \le c_3$, for some universal constant $c_3$. Hence the size of any hole is bounded above by $R \theta \le \kappa b^{-1}$  for some universal constant $\kappa$, as desired.

%This leads a contradiction, and completes the proof that $F_R$ has zero probability for $R$ large enough (depending on $\ph$).

\section{Rigidity}
\label{rigidity-sec}
In this section, we prove Theorem \ref{rigidity}. 
%As before, the arguments are different for $\Z^d$ and for $\R^d$, and will be taken up in separate sub-sections.
%\subsection{Rigidity in $\rd$}
%\label{rig-cont}
\subsection{Point processes}
\label{rig-pp}
It suffices to prove the theorem when the domain in question is a ball. This is because, any bounded set is a subset of a large enough ball, and it is easy to see that if Theorem \ref{rigidity} holds for a domain $D$, then it holds for any subset of D.  By way of notation, let $B(x;r)$ be the ball of radius $r$ and centre $x$. We want to prove the theorem for the domain $B(0;a)$. We recall the notation that $\G(x,r)$ is the cube with center $x \in \O$ and side-length $r>0$.

We consider $\mu \in \O$ and $\b>0$ small enough so that $\G(\mu,2\b) \subsetneq U$, and  set $f_\b(\xi)={1}_{\G(\mu,2\b)} \star {1}_{\G(\mu,2\b)}(\xi)$. We consider the function $\ph(x) = \h{f_\b}(x)$. Clearly, $\h{\ph}$ is supported inside $U$, and hence satisfies \eqref{fund}. Moreover, for $\th \in \rd$ such that $\|\th\|_\infty$ is small enough, we have $\G(\mu,2\b)+\th \subset U$.  Therefore, if we consider $\h{f}(\xi,\th,b):=1_{\G(\mu,2\b)+\th} \star 1_{\G(\mu,2\b)+\th} (\xi)$, and set $\ph(x,\th,\b)={f}(x,\th,\b)$, then $\h{\ph}(\cdot,\th,\b)$ is supported inside $U$ for all $\th,\b$ small enough. Hence \eqref{fund} is satisfied a.s. We will work on the event (of probability 1) where all $\ph(\cdot,\th,\b)$ (with $(\th,\b)$ as above and having all co-ordinates rational ; a set that we will denote by $Q$) satisfy \eqref{fund} simultaneously. 
 
%Arguing as in the previous section, each translate $\ph_t=\ph(\cdot - t)$ (for $t \in \rd$)  also satisfies \eqref{fund-cont}. 

We notice that $\ph$ as above has the explicit expression 
\[ \ph(x,\th,\b) = C e^{i \langle \th, x \rangle} \prod_{i=1}^d \l( \frac{\sin \b x(i)}{\b x(i)} \r)^2, \] where $x=(x(1),\cdots,x(d)) \in \rd$.
\eqref{fund} can be rewritten for $\ph(\cdot,\th,\b)$ as 
\begin{equation}
\label{rig-cont-1}
\sum_{x \in \Xi \cap B(0;a)} \ph(x,\th,\b) = \rho \h{\ph}(0,\th,\b) - \sum_{x \in \Xi \cap B(0;a)^\c} \ph(x,\th,\b).
\end{equation}

Let $X_1,\cdots,X_N$ be the points inside $B(0;a)$. 
We deduce that the point configuration $\u$ outside $B(0;a)$ determines the quantities (for each $(\th,b) \in Q$) 
\begin{equation}
\label{rig-cont-2} 
\sum_{j=1}^N  e^{i \langle \th, X_j \rangle}   \prod_{i=1}^d \l( \frac{\sin \b X_j (i)}{\b X_j(i)} \r)^2 = F(\u,\th,\b)
\end{equation}
where $F$ is a measurable function.

\eqref{rig-cont-2} is true for all $(\th,\b)$ with rational co-ordinates that lie in a small neighbourhood of the origin.  Since the left hand side of \eqref{rig-cont-2} is continuous in $(\th,\b)$ for a fixed realization $(X_1,\cdots,X_N)$, \eqref{rig-cont-2} is true for $(\th,0)$ for all $\th$ such that $(\th,\b) \in Q$. for some $\b>0$. This leads us to  
\begin{equation}
\label{rig-cont-3} 
\sum_{j=1}^N  e^{i \langle \th, X_j \rangle}   = F(\u,\th)
\end{equation}
for a measurable function $F$.

Since the left hand side of \eqref{rig-cont-3} is continuous in $\th$, therefore we can take limits to deduce that \eqref{rig-cont-3} is true for all $\th$ in a neighbourhood of $0\in \rd$. But this determines the set $(X_1,\cdots,X_N)$ completely. A simple way to see this is that the left hand side of \eqref{rig-cont-3} is the characteristic function of the (compactly supported) counting measure corresponding to the set $\{X_1,\cdots,X_N\}$.

\subsection{Random fields and random measures}
\label{rig-rm}
The techniques of Section \ref{rig-pp}, although introduced for point processes,  can be extended to the setting where $\Xi$ is a translation invariant random field, or even a translation invariant random measure on $\R^d$ or $\Z^d$. In what follows, we discuss the natural extensions of the key concepts used in Section \ref{rig-pp} to the setting of random fields and random measures.

% Indeed, we consider a situation where $\Xi$ is a translation invariant random measure on $\R^d$. 

%The strategy of proving rigidity is the same as in the case of point processes; the key point we need to make sure is that all the relevant objects make sense. 

%For a function $\ph$ and a random measure $\Xi$, we define the linear statistic \[I(\ph)=\int_{\rd} \ph \d [\Xi].\] 

%Recall from the basic definitions for random measures that 
%\[ \var[ \int_\rd  \ph(x) \d [\Xi](x) ] = \int_\rd |\h{\ph}(\xi)|^2 S(\xi)\d \xi.  \] 
%Consequently, if the structure function $S$ vanishes on a set $U \subset \rd$, and if $\ph$ is supported on $U$, then a.s. we have \[ \int_\rd \ph(x)  \d [\Xi](x) = \rho \int \ph(x) \d x . \]

Following the argument for point processes, we consider $\b>0$ small enough so that $\G(\mu,2\b) \subsetneq U$, let  $\h{f}(\xi,\th,\b):=1_{\G(\mu,2\b)+\th} \star 1_{\G(\mu,2\b)+\th} (\xi)$, and set $\ph(x,\th,\b)={f}(x,\th,\b)$. Then $\h{\ph}(\cdot,\th,\b)$ is supported inside $U$ for all $\th,\b$ small enough. 
%Hence \eqref{fund-cont} is satisfied a.s.  
Clearly, $\ph$ as above has the explicit expression 
\[ \ph(x,\th,\b) = C e^{i \langle \th, x \rangle} \prod_{i=1}^d \l( \frac{\sin \b x(i)}{\b x(i)} \r)^2, \] where $x=(x(1),\cdots,x(d)) \in \rd$.
%By choosing $\ph(x,\th,b)$ as in the case of point processes,
 
 We note that for each $(\th,\b)$ the function $\ph(\cdot,\th,\b) \in L_1(\d x)$, hence the linear statistic $\int \ph \d [\Xi]$ is well-defined and finite a.s. Exploiting the support properties of $\ph$,  we can write down an analogue of \eqref{fund} as 
\begin{equation}
\label{fund-cont-rm}
I(\ph(\cdot,\th,\b))= \int \ph(x,\th,\b) \d [\Xi] = \E[I(\ph)] = \rho \h{\ph}(0,\th,\b).
\end{equation}
We immediately obtain an analogue of \eqref{rig-cont-1} as
\begin{equation}
\label{rig-cont-1-rm}
\int_{B(0;a)} \ph(x,\th,\b) \d [\Xi] = \rho \h{\ph}(0,\th,\b) - \int_{B(0;a)^\c} \ph(x,\th,\b) \d [\Xi].
\end{equation}
Let $\Xi^{\inn}$ and $\Xi^{\out}$ respectively denote the restrictions of the random measure $\Xi$ to $B(0;a)$ and $B(0;a)^\c$. Then we want to infer about $\Xi^{\inn}$ given $\Xi^{\out}$.
Proceeding as in the case of point processes, we eventually deduce that 
\begin{equation}
\label{rig-cont-3-rm}
\int e^{i \langle \th,x \rangle} \d [\Xi^\inn] = F(\Xi^\out,\th)
\end{equation}
for all $\th$ in a neighbourhood of 0, $F$ being a measurable function. But a.s. $\Xi^\inn$ is a finite Borel measure supported on the compact set $B(0;a)$. Consequently, the characteristic function $\int e^{i \langle \th,x \rangle} \d [\Xi^\inn]$ admits an extension to a holomorphic function  in $\th \in \C^d$. But for a holomorphic function, its values in any (real) neighbourhood determines the function completely. Hence \eqref{rig-cont-3-rm} completely determines the characteristic function of $\Xi^\inn$ (for any given realization), and hence completely determines the measure $\Xi^\inn$.

\subsection{Rigidity under fast decay of the structure function}
\label{sec:fastdecay}
In this section, we show that if the structure function of an invariant point process (or more generally, an invariant random field or random measure) on $\Z^d$ or $\R^d$ vanishes at 0 faster than any polynomial, then the process in question exhibits maximal rigidity. 

Our main tool will be the following proposition:
\begin{proposition}
\label{infdec}
Let $\Xi$ be  an invariant point process (or more generally, an invariant random field or random measure) on $\Z^d$ or $\R^d$ whose structure function $S(k)$ vanishes at $k=0$ faster than any polynomial. Then, for any Schwarz function $\ph$, positive integer $m$ and $L>1$, the random integral $\int \ph(z/L) \d [\Xi](z)$ satisfies 
\begin{equation}
\label{infdeceq}
\var [\int \ph(z/L) \d [\Xi](z)] \le C(\ph,m) L^{-m},
\end{equation}
for some positive number $C(\ph,m)$.
\end{proposition}

\begin{proof}
Let $\eps_m < 1 $ be such that $|S(\xi)| \le C_m |\xi|^m$ for $|\xi| \le \eps_m$.
We recall the notation that $\ph_L(z)=\ph(z/L)$
\[ \var[ \int \ph(z/L) \d [\Xi](z)] ] = \int_{\R^d} |\hat{\ph_L}(\xi)|^2 S(\xi) \d \xi . \]
Now, $\hat{\ph_L}(\xi)=L^d\hat{\ph}(L \xi)$. Consequently, because $\ph$ is a Schwarz function, we have $|\hat{\ph_L}(\xi)| \le C(\ph,m)L^{-2d-m}(1+|\xi|)^{-2d-m}$. We estimate the integral from above by breaking it into two parts : $|\xi| \le \eps_m L^{-1}$ and $|\xi| > \eps_m L^{-1}$.
 
For $|\xi| \ge \eps_m$, we upper bound the integral by 
\[\int_{|\xi|\ge \eps_m L^{-1}} |\hat{\ph_L}(\xi)|^2 |S(\xi)| \d \xi \le C(\ph,m) L^{-m} \int_{\R^d} |S(\xi)| (1+|\xi|)^{-2d-m} \d \xi \le C(\ph,m) L^{-m} .   \]

For $|\xi| \le \eps_m$, we upper bound the integral by 
\[\int_{|\xi|\le \eps_m L^{-1}} |\hat{\ph_L}(\xi)|^2 |S(\xi)| \d \xi \le L^{2d} C(\ph,0)C_{m+d} \int_{|\xi| \le \eps_m L^{-1}} |\xi|^{d+m} \d \xi \le C'(\ph,m) L^{-m}.   \]

Combining the two bounds, we deduce the statement of the proposition.

\end{proof}

Now we are ready to prove Theorem \ref{fast-decay}.
%\begin{theorem}
%\label{infdecth}
%Let $\Xi$ be  an invariant point process (or more generally, an invariant random field or random measure) on $\Z^d$ or $\R^d$ whose structure function $S(k)$ vanishes at $k=0$ faster than any polynomial. Then $\Xi$ exhibits maximal rigidity.
%\end{theorem}

\begin{proof}[Proof of Theorem \ref{fast-decay}]
Let $D$ be a bounded domain in $\R^d$. We want to show that the point process (or the random field) outside $D$ determines the field inside $D$ with probability 1.

To this end, we consider a bump function $\ph$ that is $\equiv 1$ on $D$. Clearly, the scalings  $\ph_L(z):=\ph(z/L)$ all satisfy these properties. For $[k] \in \Z_+^d$ and $x \in \R^d$, we denote by $x^{[k]}$ the monomial $x_1^{k_1} \cdots x_d^{k_d}$, where $x_i$ and $k_i$ are respectively the co-ordinates of $x$ and $k$.  We also denote by $|k|$ the quantity $\sum_{i=1}^d k_i$. We will apply Proposition \ref{infdec} to the Schwarz functions $\ph^{[k]}(x)=x^{[k]}\ph(x)$.

For large enough $L$, depending only on $\ph,[k]$ and $\Xi$, we have 
\[\var [ \int \ph^{[k]}_L \d [\Xi] ]  \le C(\ph,k) L^{-2k-1}. \] This implies that 
for large enough $L$, depending only on $\ph,[k]$ and $\Xi$, we have 
\[\var [ \int L^{|k|}\ph^{[k]}_L \d [\Xi] ]  \le C(\ph,k) L^{-1}. \] 
Therefore, if $L$ is allowed to vary over the sequence $2^n, n \in \Z$, we can invoke the Borel Cantelli lemma and deduce that $\l(\int L^{|k|} \ph^{[k]}_L \d [\Xi]  - \E[L^{|k|} \ph^{[k]}_L \d [\Xi] ] \r)\to 0$ almost surely. Observe that
\[ \int_{\R^d} L^{|k|} \ph^{[k]}_L \d [\Xi]  =  \int_{D} L^{|k|} \ph^{[k]}_L \d [\Xi] +  \int_{D^\c} L^{|k|} \ph^{[k]}_L \d [\Xi]. \]
  But on the set $D$,  the function $L^{|k|} \ph^{[k]}_L (x)$ is simply $x^{[k]}$. This enables us to deduce that
  \[ \l(\int_D x^{[k]} \d [\Xi] +  \int_{D^\c} L^{|k|} \ph^{[k]}_L \d [\Xi] - \E \l[L^{|k|} \ph^{[k]}_L \d [\Xi] \r] \r) \to 0 \] almost surely.  But both the second and the third terms in the display above can be computed exactly, given the field outside $D$. This completes the proof of maximal rigidity.
\end{proof}

\section{Origins of maximal rigidity : Gaussian processes on $\Z^d$}
\label{sec:Gaussian process}
To explore the origin of the maximal rigidity imposed by the gap, we consider generalized stealthy (GS) Gaussian processes on $\Z^d$. We will consider such a process as a limit of finite ``stealthy'' Gaussian processes. By a finite  Gaussian process, we mean a Gaussian process on $\Z^d_n$; its wave space being identified with the discrete torus $\t^d_n$. In this context, ``GS'' would imply the structure function vanishing on a subset $U_n \subset \t^d_n$. The idea is that, for large $n$, $\t^d_n$ approximates $\t^d$ and $U_n$ approximates a neighbourhood in $\td$. 

Let $\l(\xi_n(x)\r)_{x \in \t^d_n}$ be a Gaussian random variable with covariance
\[ C(x-x')=\sum_{k=0}^{n-1}S(\frac{2 \pi i k}{n})\exp\l(\frac{2 \pi i k}{n}(x-x')\r).  \] It follows that this Gaussian field (that lives on $\Z^d_n$) has the structure function $S(\cdot)$ (that lives on $\t^d_n$).  It also follows that this Gaussian random variable has a density (with respect to the Lebesgue measure on $\R^n$) that is given by
\[  \mu(\xi)= \frac{1}{(2 \pi)^{d/2} |C|^{1/2}} \exp\l( -\frac{1}{2}  \sum_{x,x' \in \t_n^d} \xi(x) C^{-1}(x-x') \xi(x')\r)\]
Let $U$ be a fixed open set in $\t^d$.  We choose $S$ to be a non-negative function on $\t^d$ that vanishes on $U$; note that by the Bochner correspondence, such a function gives rise to a (unique up to translations) Gaussian process on $\Z^d$. To emulate a (generalized) stealthy Gaussian process in the finite setting, we simply consider the restriction of $S$ to $\t_n^d$ (identified canonically as a subset of $\t^d$).

We now define the Gaussian vector $\wt{\xi}$ (on $\t^d_n$) by 
\[ \wt{\xi}(k) = \frac{1}{n^{d/2}} \sum_{x \in \Z^d_n} \xi(x) \exp(i kx), \] where $k=\frac{2\pi j}{n}, j \in \Z_n^d$. Observe that for each $k$, this is in particular, a linear combination of $\xi(x)_{x \in \Z^d_n}$. Using the Fourier transform, a simple computation shows that \[ \E[|\wt{\xi}(k)|^2]=S(k),\] and hence $ \E[|\wt{\xi}(k)|^2]=0$ for $k \in U$, which implies $\wt{\xi}(k)=0$ a.s. for such $k$. Since the number of such indices $k$ is $|U|\cdot n^d$, we conclude that the condition of being GS imposes $|U|\cdot n^d$ linear constraints on the Gaussian process $\xi$ living on the physical space $\Z^d_n$. Because of these linear constraints, the values of the Gaussian process on any domain (in $\Z^d_n$) of size $|U|\cdot n^d$ is determined exactly by the rest of the values. This gives a version of maximal rigidity for finite $n$.

Thus, for finite $n$, a GS Gaussian process exhibits degeneracy (i.e., mutual  algebraic dependence of its values), and the degree of this degeneracy is, heuristically speaking, proportional to the size of the gap in the spectrum. The maximal rigidity can be viewed as a limit of this degeneracy in the finite situation.   

\section{Acknowledgements}
The work of J.L.L. was supported in part by  the AFOSR grant FA9550-16-1-0037. The work of S.G. was supported in part by the ARO grant W911NF-14-1-0094. J.L.L. thanks the Systems Biology group of the Institute for Advanced Studies for their hospitality during this work. The authors thank Salvatore Torquato for valuable discussions.

\end{document}